\documentclass[12pt,a4paper]{amsart}

\usepackage[utf8]{inputenc}
\usepackage[T1]{fontenc}
\usepackage[english]{babel}
\usepackage{cite}

\usepackage{indentfirst}
\usepackage{amssymb}
\usepackage{amsfonts}
\usepackage{amsmath}
\usepackage{amsthm}
\usepackage[top=2.5cm, bottom=2.5cm, left=2.5cm, right=2.5cm]{geometry}
\usepackage{amsopn}
\usepackage[colorlinks]{hyperref}
\usepackage{tikz}
\usetikzlibrary{patterns}
\usetikzlibrary{arrows}

\theoremstyle{plain}
\newtheorem{thm}{Theorem}

\newtheorem{lem}[thm]{Lemma}
\newtheorem{prop}[thm]{Proposition}
\newtheorem{cor}[thm]{Corollary}

\theoremstyle{definition}

\newtheorem*{example*}{Example}

\newtheorem*{rem*}{Remark}

\newcommand{\FF}{\phantom{ }_{2}F_{1}} 
\newcommand{\BB}{\mathrm{B}} 
\newcommand{\al}{\alpha}
\newcommand{\be}{\beta}

\newcommand{\R}{\mathbb{R}}
\newcommand{\C}{\mathbb{C}}

\DeclareMathOperator{\supp}{supp}

\DeclareMathOperator{\dist}{dist}
\DeclareMathOperator{\Ree}{Re}

\DeclareMathOperator{\fgamm}{{\boldsymbol\gamma}}

\DeclareSymbolFont{bbsymbol}{U}{bbold}{m}{n}
\DeclareMathSymbol{\ind}{\mathbin}{bbsymbol}{'061}
\title{Sharp weighted fractional Hardy inequalities}
\author[B{.} Dyda]{Bart{\l}omiej Dyda}
\author[M{.} Kijaczko]{Micha\l{} Kijaczko}
\keywords{fractional Hardy inequality, weight, best constant, half-space, convex domain}
\subjclass[2020]{46E35, 39B72, 26D15}
\address[B.D. and M.K.]{Faculty of Pure and Applied Mathematics\\ Wroc{\l}aw University 
	of Science and Technology\\
	Wybrze\.ze Wyspia\'nskiego 27,
	50-370 Wroc{\l}aw, Poland
}
\email{bdyda@pwr.edu.pl\qquad dyda@math.uni-bielefeld.de}
\email{michal.kijaczko@pwr.edu.pl}
\begin{document}

\begin{abstract}
We investigate the weighted fractional order Hardy inequality 
$$
\int_{\Omega}\int_{\Omega}\frac{|f(x)-f(y)|^{p}}{|x-y|^{d+sp}}\dist(x,\partial\Omega)^{-\alpha}\dist(y,\partial\Omega)^{-\beta}\,dy\,dx\geq C\int_{\Omega}\frac{|f(x)|^{p}}{\dist(x,\partial\Omega)^{sp+\alpha+\beta}}\,dx,    
$$
for  $\Omega=\R^{d-1}\times(0,\infty)$, $\Omega$ being a convex domain or $\Omega=\R^d\setminus\{0\}$. Our work focuses on finding the best (i.e. sharp) constant $C=C(d,s,p,\alpha,\beta)$ in all cases. We also obtain weighted version of the fractional Hardy--Sobolev--Maz'ya inequality.
The proofs are based on general Hardy inequalities and the non-linear ground state representation, established by Frank and Seiringer.
\end{abstract}

\maketitle

\section{Introduction and main results}

In this paper we consider fractional order weighted Hardy inequalities,
\begin{equation}\label{generalweightedfractionalHardy}
\int_{\Omega}\int_{\Omega}\frac{|u(x)-u(y)|^{p}}{|x-y|^{d+sp}}\dist(x,\partial\Omega)^{-\alpha}\dist(y,\partial\Omega)^{-\beta}\,dy\,dx\geq C\int_{\Omega}\frac{|f(x)|^{p}}{\dist(x,\partial\Omega)^{sp+\alpha+\beta}}\,dx, 
\end{equation}
where $u\in C_{c}(\Omega)$. Here $\Omega$ is a nonempty, proper, open subset of $\R^d$ and $\dist(x,\partial\Omega)=\inf_{y\in\partial\Omega}|x-y|$ denotes the distance to the boundary.
We are interested in obtaining the sharp constant $C$ in the inequality \eqref{generalweightedfractionalHardy}, under some assumptions on $\Omega$ and the other parameters. 

Fractional order Hardy inequalities are of great interest in last decades, due to their connection with probability theory (especially Lévy processes) or partial differential equations. For (non-weighted) fractional Hardy inequalities with best constants, we refer to \cite{bbz, MR2469027, FLS, MR2663757, MR2723817, LossSloane}.
We also refer to  \cite{MR3420496} and \cite{MK} for a parallel topic of weighted fractional Sobolev spaces related to the left-hand side of \eqref{generalweightedfractionalHardy}.

\subsection{Weighted fractional Hardy inequalities for the half-space}
Our main goal is to obtain the results for the fractional weighted Hardy inequality in the half-space $\R^d_+:=\{(x_1,x_2,\dots,x_d)\in\R^d:x_d>0\}$. The unweighted case was studied by Bogdan and Dyda \cite{MR2663757} for $p=2$ and generalized by Frank and Seiringer \cite{MR2723817} to arbitrary $p\geq 1$. According to our best knowledge, the weighted case in new.

\begin{thm}[\textbf{Sharp weighted fractional Hardy inequality on $\R^d_+$}]\label{tw3}
Let $0<s<1$, $p\geq 1$, $\alpha,\beta,\alpha+\beta\in (-1,sp)$ and $1+\alpha+\beta\neq sp$. Then for all $u\in C_{c}(\R^{d}_{+})$ the following inequality holds, 
\begin{equation}\label{hardy2}
\int_{\R^{d}_{+}}\int_{\R^{d}_{+}}\frac{|u(x)-u(y)|^{p}}{|x-y|^{d+sp}}\,x_{d}^{\alpha}\,y_{d}^{\beta}\,dy\,dx\geq\mathcal{D}\int_{\R^{d}_{+}}\frac{|u(x)|^{p}}{x_{d}^{sp-\alpha-\beta}}\,dx, 
\end{equation}
where $\mathcal{D}>0$
is the optimal constant given by \eqref{D}.
\end{thm}

If we additionally assume that $p\geq 2$, then we obtain the following stronger inequality with a~remainder.

\begin{thm}[\textbf{Sharp weighted fractional Hardy inequality on $\R^d_+$ with remainder}]\label{tw4}
Let $0<s<1$, $p\geq 2$, $\alpha,\beta,\alpha+\beta\in (-1,sp)$ and $1+\alpha+\beta\neq sp$. Then for all $u\in C_{c}(\R^{d}_{+})$ the following inequality holds, 
\begin{align*}
\int_{\R^{d}_{+}}\int_{\R^{d}_{+}}\frac{|u(x)-u(y)|^{p}}{|x-y|^{d+sp}}\,x_{d}^{\alpha}\,y_{d}^{\beta}\,dy\,dx-&\mathcal{D}\int_{\R^{d}_{+}}\frac{|u(x)|^{p}}{x_{d}^{sp-\alpha-\beta}}\,dx\\
&\geq c_{p}\int_{\R^{d}_{+}}\int_{\R^{d}_{+}}\frac{|v(x)-v(y)|^{p}}{|x-y|^{d+sp}}x_{d}^{-\frac{1-\alpha+\beta-sp}{2}}y_{d}^{-\frac{1+\alpha-\beta-sp}{2}}\,dy\,dx,
\end{align*}
where $v(x)=x_{d}^{\frac{1+\alpha+\beta-sp}{p}}u(x)$, $c_{p}$ is given by
\begin{equation}\label{cp}
c_{p}=\min_{0<\tau<\frac{1}{2}}\left((1-\tau)^{p}-\tau^{p}+p\tau^{p-1}\right)   
\end{equation}
and $\mathcal{D}>0$ is a constant from \eqref{D}. For $p=2$ this is an equality with $c_{2}=1$.
\end{thm}

Theorem~\ref{tw4}  also gives us the information that the inequality \eqref{hardy2} is strict, unless $u$ is identically zero. 
As an application of Theorem~\ref{tw4}, we obtain the following version of the weighted fractional Hardy--Sobolev--Maz'ya inequality for the half-space.
\begin{thm}[\textbf{Fractional weighted Hardy--Sobolev--Maz'ya inequality on $\R^d_+$}]\label{tw5}
Let $0<s<1$, $p\geq 2$, $sp<d$, $\alpha,\beta,\alpha+\beta\in (-1,sp)$ and $1+\alpha+\beta\neq sp$. Then for all $u\in C_{c}(\R^{d}_{+})$ the following inequality holds, 
\begin{align*}
 \int_{\R^{d}_{+}}\int_{\R^{d}_{+}}\frac{|u(x)-u(y)|^{p}}{|x-y|^{d+sp}}\,x_{d}^{\alpha}\,y_{d}^{\beta}\,dy\,dx-&\mathcal{D}\int_{\R^{d}_{+}}\frac{|u(x)|^{p}}{x_{d}^{sp-\alpha-\beta}}\,dx\\
 &\geq C\left(\int_{\R^d_+}|u(x)|^{q}x_{d}^{\frac{q}{p}(\al+\be)}\,dx\right)^{\frac{p}{q}},
\end{align*}
where $C=C(\al,\be,d,s,p)>0$ is a constant and $q=\frac{dp}{d-sp}$.
\end{thm}

\subsection{Weighted fractional Hardy inequalities on convex domains}
We generalize the result obtained by Loss and Sloane \cite{LossSloane} to the weighted setting, that is we prove that the inequality \eqref{generalweightedfractionalHardy} holds for all convex domains $\Omega\subset\R^d$ with the same optimal constant~$\mathcal{D}$ as for the half-space. More precisely, we have the following result.
\begin{thm}[\textbf{Weighted fractional Hardy inequality for general domains}]\label{tw7}
Let $0<s<1$, $p\geq 1$, $\al,\be,\al+\be\in(-1,0]$ and $sp>1+\al+\be$. Let $\Omega\subset\R^d$ be a nonempty, proper, open set and denote $d_\Omega(x)=\dist(x,\partial \Omega)$. Then, for all $u\in C_c(\Omega)$,
\begin{equation}\label{hardyconvex1}
\int_\Omega\int_\Omega\frac{|u(x)-u(y)|^{p}}{|x-y|^{d+sp}}d_\Omega(x)^{\al}d_{\Omega}(y)^{\be}\,dy\,dx\geq\mathcal{D}\int_{\Omega}\frac{|u(x)|^{p}}{m_{sp-\al-\be}(x)^{sp-\al-\be}}\,dx,    
\end{equation}
where $\mathcal{D}=\mathcal{D}(d,s,p,\al,\be)$ is given by formula \eqref{D},
\begin{equation}\label{m}
m_{a}(x)^{a}=\frac{\int_{\mathbb{S}^{d-1}}|\omega_{d}|^{a}\,d\omega}{\int_{\mathbb{S}^{d-1}}d_{\omega,\Omega}(x)^{-a}\,d\omega}, \quad d_{\omega,\Omega}(x)=\min\{|t|:x+t\omega\notin \Omega\}.
\vspace{0.1cm}
\end{equation}
In particular, if $\Omega$ is convex, then
\begin{equation}\label{hardyconvex2}
\int_\Omega\int_\Omega\frac{|u(x)-u(y)|^{p}}{|x-y|^{d+sp}}d_\Omega(x)^{\al}d_{\Omega}(y)^{\be}\,dy\,dx\geq\mathcal{D}\int_{\Omega}\frac{|u(x)|^{p}}{d_\Omega(x)^{sp-\al-\be}}\,dx.    
\end{equation}
The constant in \eqref{hardyconvex2} is sharp.
\end{thm}

\subsection{Weighted fractional Hardy inequalities on $\R^d$}
Next, we focus on weighted fractional Hardy inequalities for the full space $\R^d$ and this is our  second main result.
\begin{thm}[\textbf{Sharp weighted fractional Hardy inequality on $\R^d$}]\label{tw1}
  Let $0<s<1$, $p\geq 1$ and
  $\alpha, \beta, \alpha+\beta \in (-sp,d)$.
  Then for all $u\in C_{c}(\R^d)$, when $sp+\al+\be<d$ and for all $u\in C_{c}(\R^d\setminus\{0\})$, when $sp+\al+\be>d$, the following inequality holds, 
\begin{equation}\label{hardy1}
\int_{\R^d}\int_{\R^d}\frac{|u(x)-u(y)|^{p}}{|x-y|^{d+sp}}\,|x|^{-\alpha}\,|y|^{-\beta}\,dy\,dx\geq\mathcal{C}\int_{\R^d}\frac{|u(x)|^{p}}{|x|^{sp+\alpha+\beta}}\,dx, 
\end{equation}
where $\mathcal{C}$
is the optimal constant given by \eqref{C}.
\end{thm}

For an arbitrary $p\geq 1$, the case $\al=\be=0$ of the inequality \eqref{hardy1} was obtained by Frank and Seiringer in \cite{MR2469027}, using the non-linear ``ground state" representation. Abdellaoui and Bentifour \cite{MR3626031} proved a version of \eqref{hardy1} for $\al=\be \in[0,(d-sp)/2)$.
The inequality \eqref{hardy1} in the unweighted case, i.e., for $\alpha=\beta=0$, and $p=2$,  was obtained first by Herbst \cite{MR0436854} and independently by Yafaev \cite{MR1717839} and Beckner \cite{MR2373619}. The Reader may also see \cite{FLS} for a direct computation via Fourier transform.

Again as for the half-space, for $p\geq 2$, the inequality \eqref{hardy1} holds with a remainder.

\begin{thm}[\textbf{Sharp weighted fractional Hardy inequality on $\R^d$ with remainder}]\label{tw2}
  Let $0<s<1$, $p\geq 2$ and
   $\alpha, \beta, \alpha+\beta \in (-sp,d)$.
   Then for all $u\in C_{c}(\R^d)$,  when $sp+\al+\be<d$ and for all $u\in C_{c}(\R^d\setminus\{0\})$, when $sp+\al+\be>d$, the following inequality holds,
\begin{align*}
\int_{\R^d}\int_{\R^d}\frac{|u(x)-u(y)|^{p}}{|x-y|^{d+sp}}\,|x|^{-\alpha}\,|y|^{-\beta}&\,dy\,dx-\mathcal{C}\int_{\R^d}\frac{|u(x)|^{p}}{|x|^{sp+\alpha+\beta}}\,dx\\
&\geq c_{p}\int_{\R^d}\int_{\R^d}\frac{|v(x)-v(y)|^{p}}{|x-y|^{d+sp}}\,|x|^{-\frac{d+\alpha-\beta-sp}{2}}\,|y|^{-\frac{d-\alpha+\beta-sp}{2}}\,dy\,dx,
\end{align*}
where $v(x)=|x|^{\frac{d-\alpha-\beta-sp}{p}}u(x)$, $\mathcal{C}$ is a constant from \eqref{C} and $c_p$ is a constant from \eqref{cp}. For $p=2$ this is an equality with $c_{2}=1$.
\end{thm}

The above inequality was proved for $\alpha=\beta=0$ in \cite{MR2469027},
and for $\alpha=\beta$ in  \cite{MR3626031}. In the latter paper, the Authors also considered
different forms of the remainder.

An obvious consequence of Theorem \ref{tw2} is that the Hardy inequality \eqref{hardy1} is strict for any nonzero function $u$, if $p\geq 2$.

In contrast to the unweighted case, it turns out that the weighted Gagliardo seminorms may be finite also for $s=0$ and $C_c^\infty$ functions, see for example \cite[Proof of Lemma 2.1]{MR3420496} for the proof on $\R^d$ (the proof for the half-space works analogously). In consequence, by passing to the limit with $s\rightarrow 0^+$ in \eqref{hardy2} and \eqref{hardy1}, we obtain the following weighted fractional Hardy inequalities related to the $0$-order kernel $|x-y|^{-d}$.
\begin{cor}[\textbf{Weighted fractional Hardy inequalities for $s=0$}]\label{corollary7}
Let $p\geq 1$, $\al,\be,\al+\be\in(0,d)$  and $u\in C_c(\R^d)$. Then, the following inequality holds true,
\begin{equation}\label{hardy_s=0_Rd}
\int_{\R^d}\int_{\R^d}\frac{|u(x)-u(y)|^p}{|x-y|^{d}}|x|^{-\al}|y|^{-\be}\,dy\,dx\geq\mathcal{C}_0\int_{\R^d}\frac{|u(x)|^p}{|x|^{\al+\be}}\,dx,    
\end{equation}
where $\mathcal{C}_0 = \mathcal{C}(d,0,p,\al,\be)$ is given by \eqref{C} or \eqref{C0}.

Moreover, if $\al,\be,\al+\be\in(-1,0)$ and $u\in C_c(\R^d_+)$, then
\begin{equation}\label{hardy_s=0_Rd+}
\int_{\R^d_+}\int_{\R^d_+}\frac{|u(x)-u(y)|^p}{|x-y|^{d}}x_d^{\al}y_d^{\be}\,dy\,dx\geq\mathcal{D}_0\int_{\R^d_+}|u(x)|^p x_d^{\al+\be}\,dx,    
\end{equation}
where $\mathcal{D}_0 = \mathcal{D}(d,0,p,\al,\be)$ is given by \eqref{D}.
\end{cor}
Considering the limit $s\rightarrow 0^+$ in the appriopriate inequalities, one may also easily obtain a versions of \eqref{hardy_s=0_Rd} and \eqref{hardy_s=0_Rd+} with a remainder for $p\geq 2$.

\section{General Hardy inequalities}\label{sec:generalhardy}
The proofs of our results are based on the general fractional Hardy inequalities and the non-linear ground state representation, established by Frank and Seiringer in \cite{MR2469027}. Following \cite{MR2469027}, let $k(x,y)$ be a symmetric, positive kernel and $\Omega\subset\R^d$ be an open, nonempty set. Let us define the functional 
\[
E[u]=\int_{\Omega}\int_{\Omega}|u(x)-u(y)|^{p}k(x,y)\,dy\,dx
\]
and 
\begin{equation}\label{eq:plaplace}
V_{\varepsilon}(x)=2w(x)^{-p+1}\int_\Omega\left(w(x)-w(y)\right)\left|w(x)-w(y)\right|^{p-2}k_{\varepsilon}(x,y)\,dy,
\end{equation}
where $w$ is a positive, measurable function on $\Omega$ and $\{k_{\varepsilon}(x,y)\}_{\varepsilon>0}$ is a family of measurable, symmetric kernels satisfying the assumptions $0\leq k_\varepsilon(x,y)\leq k(x,y)$, $\lim_{\varepsilon\rightarrow 0}k_\varepsilon(x,y)=k(x,y)$ for all $x,y\in\Omega$. Then, if the integrals defining $V_\varepsilon(x)$ are absolutely convergent for almost all $x\in\Omega$ and converge weakly to some $V$ in $L^{1}_{loc}(\Omega)$, we have the following Hardy-type inequality, 
\begin{equation}\label{generalhardy}
 E[u]\geq\int_{\Omega}|u(x)|^{p}V(x)\,dx,   
\end{equation}
for all compactly supported $u$ with $\int_{\Omega}|u(x)|^{p}V_{+}(x)\,dx$ finite, see \cite[Proposition 2.2]{MR2469027}.

Moreover, if in addition $p\geq 2$, then the inequality \eqref{generalhardy} can be improved by a remainder
\begin{equation}
E_{w}[v]=\int_{\Omega}\int_{\Omega}|v(x)-v(y)|^{p}w(x)^{\frac{p}{2}}k(x,y)w(y)^{\frac{p}{2}}\,dy\,dx,\quad u=wv,  
\end{equation}
that is the inequality
\begin{equation}\label{generalhardyremainder}
 E[u]-\int_{\Omega}|u(x)|^{p}V(x)\,dx\geq c_{p}E_{w}[v]  
\end{equation}
holds with the same assumptions as in \eqref{generalhardy}, with the constant $c_p$ from \eqref{cp}, see \cite[Proposition 2.3]{MR2469027}. When $p=2$, the inequality \eqref{generalhardyremainder} becomes an equality.

\section{Fractional weighted $p$-Laplacian and power functions}

We saw in Section~\ref{sec:generalhardy} that in order to obtain an explicit
Hardy inequality, it suffices to calculate the limit of $V_\varepsilon$ defined in \eqref{eq:plaplace}.
This section is devoted to such calculations.

By $\Gamma$ we denote the Euler Gamma function. We consider the Gamma function to be defined on $\C\setminus \{0,-1,\ldots\}$, but $1/\Gamma$ to be defined everywhere on $\C$, with zeroes at $0,-1,\ldots$. Then Euler Beta function may be defined as $\BB(a,b)=\Gamma(a)\Gamma(b)(1/\Gamma)(a+b)$ for all $a,b\neq 0,-1,\ldots$.

\subsection{The case of $\R^d\setminus \{0\}$}

In this subsection we will denote $$\gamma=\frac{d-\alpha-\beta-sp}{p},\,w(x)=|x|^{-\gamma},\,k(x,y)=\frac{1}{2}|x-y|^{-d-sp}\left(|x|^{-\alpha}|y|^{-\beta}+|x|^{-\beta}|y|^{-\alpha}\right).
$$
Furthermore, we put
\begin{equation}\label{C}
\mathcal{C}=\mathcal{C}(d,s,p,\alpha,\beta)=\int_{0}^{1}r^{sp-1}\left(r^{\alpha}+r^{\beta}\right)\left|1-r^{\frac{d-\alpha-\beta-sp}{p}}\right|^{p}\Phi_{d,s,p}(r)\,dr.
\end{equation}
The function $\Phi_{d,s,p}$ is defined as 
\begin{equation}\label{Phi}
\Phi_{d,s,p}(r)=\begin{cases}
\left|\mathbb{S}^{d-2}\right|\displaystyle\int_{-1}^{1}\frac{\left(1-t^{2}\right)^{\frac{d-3}{2}}}{\left(1-2tr+ r^2
  \right)^{\frac{d+sp}{2}}}\,dt,\,d\geq 2 \\
(1-r)^{-1-sp}+(1+r)^{-1-sp},\,d=1.
\end{cases}
\end{equation}

Following \cite{MR2469027}, we remark here that for $d\geq 2$, by \cite[(3.665)]{MR2360010} the function from \eqref{Phi} can also written in the form
\begin{equation}\label{PhiFF}
  \Phi_{d,s,p}(r)=\left|\mathbb{S}^{d-1}\right| \FF\left(\frac{d+sp}{2},\frac{2+sp}{2};\frac{d}{2};r^2\right),
\end{equation}
where $\BB$ is the Euler Beta function and $\FF$ is the Gauss hypergeometric function.
We also note that
\begin{equation}\label{eq:Phibound}
  (1-r)^{sp+1}\Phi_{d,s,p}(r) \quad \textrm{is bounded for $r\in (0,1)$,}
\end{equation}
which was shown in \cite[between (3.4) and (3.5)]{MR2469027}.

\begin{lem}\label{lem.w1}
 Let  $\alpha, \beta, \alpha+\beta \in (-sp,d)$.
It holds
\begin{equation}\label{eq:VeRd}
 2\lim_{\varepsilon\rightarrow 0}\int_{||x|-|y||>\varepsilon}\left(w(x)-w(y)\right)|w(x)-w(y)|^{p-2}k(x,y)\,dy=\frac{\mathcal{C}(d,s,p,\alpha,\beta)}{|x|^{sp+\alpha+\beta}}w(x)^{p-1}
 \end{equation}
 uniformly on compacts sets contained in $\R^{d}\setminus\{0\}$. The constant $\mathcal{C}(d,s,p,\al,\be)$ is given by \eqref{C}.
\end{lem}
\begin{proof}
The proof follows \cite[Proof of Lemma 3.1]{MR2469027} and also \cite[Proof of Lemma 2.6]{MR3626031}, nevertheless we provide it for Reader's convenience. We observe that
\begin{align}\label{eq:Ieps}
2\int_{||x|-|y||>\varepsilon}\left(w(x)-w(y)\right)|w(x)-w(y)|^{p-2}k(x,y)\,dy=|x|^{-\alpha}I_{\varepsilon}(\beta)+|x|^{-\beta}I_{\varepsilon}(\alpha),   
\end{align}
where, for $a\in\{-\alpha,-\beta\}$ and $r=|x|$,
$$
I_{\varepsilon}(a)=r^{-d+1}\int_{|\rho-r|>\varepsilon}\frac{\text{sgn}\left(\rho^{\gamma}-r^{\gamma}\right)}{|\rho-r|^{2-p(1-s)}}\varphi(\rho,r)\,\rho^{a}\,d\rho
$$
with
$$
\varphi(\rho,r)=\left|\frac{\rho^{-\gamma}-r^{-\gamma}}{\rho-r}\right|^{p-1}\times\begin{cases}
\rho^{d-1}\left(1-\frac{\rho}{r}\right)^{1+sp}\Phi\left(\frac{\rho}{r}\right),\,\rho<r\\
r^{d-1}\left(1-\frac{r}{\rho}\right)^{1+sp}\Phi\left(\frac{r}{\rho}\right),\,\rho>r,
\end{cases}
$$
where $\Phi=\Phi_{d,s,p}$ is given by (\ref{Phi}).
 The convergence of the integral in \eqref{eq:Ieps} (at $0$ and at $\infty$) follows from
  the assumptions $\alpha, \beta, \alpha+\beta \in (-sp,d)$.
By exactly the same arguments
as in \cite{MR2469027} we have
\begin{align*}
I_{\varepsilon}(a)=r^{-\gamma(p-1) -sp + a}\int_{|\rho-1|>\varepsilon}\frac{\text{sgn}\left(\rho^{\gamma}-1\right)}{|\rho-1|^{2-p(1-s)}}\varphi(\rho,1)\,\rho^{a}\,d\rho    
\end{align*}
and 
\begin{align*}
  \lim_{\varepsilon\rightarrow 0}I_{\varepsilon}(a)&=r^{-\gamma(p-1) -sp +a}
 \, \lim_{\varepsilon\to 0} \left( \int_0^{1-\varepsilon} + \int_{1+\varepsilon}^\infty \right) 
\frac{\text{sgn}\left(\rho^{\gamma}-1\right)}{|\rho-1|^{2-p(1-s)}}\varphi(\rho,1)\,\rho^{a}\,d\rho\\
&=r^{-\gamma(p-1) -sp +a}\,\text{sgn}(\gamma)\int_{0}^{1}\rho^{sp-1}\Phi(\rho)\left(\rho^{-a}-\rho^{ -\gamma(p-1)+a-sp+d}\right)\left|1-\rho^{\gamma}\right|^{p-1}\,d\rho.
\end{align*}
 The last integral is convergent because of the assumptions on  $\alpha$, $\beta$ and the bound~\eqref{eq:Phibound}.
The final form of the constant $\mathcal{C}(d,s,p,\alpha,\beta)$ follows from the easy calculation
\begin{align*}
  &
 t^{sp-1}
  \text{sgn}(\gamma)\left(t^{\alpha}-t^{ -\gamma(p-1)-\alpha-sp+d}+t^{\beta}-t^{ -\gamma(p-1)-\beta-sp+d}\right)\left|1-t^{\gamma}\right|^{p-1}\\
&=t^{sp-1}\left(t^{\alpha}+t^{\beta}\right)\left|1-t^{\gamma}\right|^{p}.\qedhere
\end{align*}
\end{proof}

For $p=2$ the  constant \eqref{C} can be written explicitly, as shown in the following result.

\begin{prop}\label{prop:C}
Let  $\alpha, \beta, \alpha+\beta \in (-2s,d)$  and $0<s<1$.
  It holds
  
\begin{align}\label{Cds2}
\mathcal{C}(d,s,2,\al,\be)=\frac{\pi^{\frac{d}{2}}\left|\Gamma(-s)\right|}{\Gamma\left(\frac{d+2s}{2}\right)}\Bigg[&\frac{2\,\Gamma\left(\frac{d-\al+\be+2s}{4}\right)\Gamma\left(\frac{d+\al-\be+2s}{4}\right)}{\Gamma\left(\frac{d+\al-\be-2s}{4}\right)\Gamma\left(\frac{d-\al+\be-2s}{4}\right)}\\
  &
  \nonumber
  -\frac{\Gamma\left(\frac{\al+2s}{2}\right)\Gamma\left(\frac{d-\al}{2}\right)}{\Gamma\left(\frac{d-\al-2s}{2}\right)\Gamma\left(\frac{\al}{2}\right)}-\frac{\Gamma\left(\frac{\be+2s}{2}\right)\Gamma\left(\frac{d-\be}{2}\right)}{\Gamma\left(\frac{d-\be-2s}{2}\right)\Gamma\left(\frac{\be}{2}\right)}\Bigg].
\end{align}

 Noteworthly, for $\al=\be$ we have
$$
\mathcal{C}(d,s,2,\al,\al)=\frac{2\pi^{\frac{d}{2}}\left|\Gamma(-s)\right|}{\Gamma\left(\frac{d+2s}{2}\right)}\left[\frac{\Gamma^{2}\left(\frac{d+2s}{4}\right)}{\Gamma^{2}\left(\frac{d-2s}{4}\right)}-\frac{\Gamma\left(\frac{\al+2s}{2}\right)\Gamma\left(\frac{d-\al}{2}\right)}{\Gamma\left(\frac{d-\al-2s}{2}\right)\Gamma\left(\frac{\al}{2}\right)}\right].
$$
In particular, when $\al=\be=0$, one resolves the constant from classical unweighted case  of fractional Hardy inequality, 
\[
\mathcal{C}(d,s,2,0,0)=\frac{2\pi^{\frac{d}{2}}\left|\Gamma(-s)\right|\Gamma^{2}\left(\frac{d+2s}{4}\right)}{\Gamma\left(\frac{d+2s}{2}\right)\Gamma^{2}\left(\frac{d-2s}{4}\right)}.
\]
\end{prop}

\begin{proof}
First, we assume that
 $2s<d$ and
    $\al,\be,\al+\be\in(0, d-2s)$.  We are going to prove the result by calculating the left hand side of \eqref{eq:VeRd}
  in another way.   
Let $\Omega$ be an open subset of $\R^d$. Recall that the \emph{regional fractional Laplacian} is defined for $u\in C^2(\Omega)$ as 
$$
\Delta^{s}_\Omega u(x)= \mathcal{A}_{d,-2s} \lim_{\varepsilon\to 0^+}
\int_{\Omega\cap \{|y-x|>\varepsilon\}}
\frac{u (y)-u(x)}{|x-y|^{d+2s}} \,dy\,, 
$$  
where $\mathcal{A}_{d,-2s}=\frac{4^{s}\Gamma(\frac{d+2s}{2})}{\pi^{d/2}|\Gamma(-s)|}$. We will use the notation $L = \mathcal{A}_{d,-2s}^{-1} \Delta^{s}_\Omega.$

Let $w(x)=|x|^{-b}$ for $b\in (0,d-2s)$. By \cite[(3.3) and (3.4)]{FLS}, the value of the fractional Laplace operator $L=L_{\R^d}$ acting on the function $w$ is
\begin{equation}\label{modul}
 Lw(x)=-\frac{\pi^{\frac{d}{2}}\Gamma\left(\frac{b+2s}{2}\right)\Gamma\left(\frac{d-b}{2}\right)\left|\Gamma(-s)\right|}{\Gamma\left(\frac{d+2s}{2}\right)\Gamma\left(\frac{d-b-2s}{2}\right)\Gamma\left(\frac{b}{2}\right)}|x|^{-b-2s}.   
\end{equation}
Denote 
$$
\lambda_{d,b,s}:=\frac{\Gamma\left(\frac{b+2s}{2}\right)\Gamma\left(\frac{d-b}{2}\right)}{\Gamma\left(\frac{d-b-2s}{2}\right)\Gamma\left(\frac{b}{2}\right)}
$$
and
\begin{equation}\label{Lambda}
\Lambda_{d,b,\al,\be,s}=\lambda_{d,b+\beta,s}+\lambda_{d,b+\alpha,s}-\lambda_{d,\beta,s}-\lambda_{d,\alpha,s}.
\end{equation}

Let us notice that
\begin{align}
\left(|x|^{-b}-|y|^{-b}\right)&\left(|x|^{-\al}|y|^{-\be}+|x|^{-\be}|y|^{-\al}\right)\nonumber\\
&=|x|^{-\al}\left(|x|^{-b-\be}-|y|^{-b-\be}\right)+|x|^{-\be}\left(|x|^{-b-\al}-|y|^{-b-\al}\right)\nonumber\\
&\quad-|x|^{-\al-b}\left(|x|^{-\be}-|y|^{-\be}\right)-|x|^{-\be-b}\left(|x|^{-\al}-|y|^{-\al}\right).
\label{eq:xaxb}
\end{align}
Hence, taking $b=\gamma$, multiplying both sides by $|x-y|^{-d-2s}$, integrating over $dy$ and using \eqref{modul}
(four times, for $\gamma+\beta$, $\gamma+\alpha$, $\beta$ and $\alpha$ in place of $b$;
  note that here we need the assumption that 
 $\alpha,\beta \in (0,d-2s)$),
we obtain that
\begin{align}\label{JJ}
\nonumber
& 2\cdot \text{P.V.} \int_{\R^d}\left(|x|^{-\gamma}-|y|^{-\gamma}\right)\frac{|x|^{-\al}|y|^{-\be}+|x|^{-\be}|y|^{-\al}} 2|x-y|^{-d-2s}\,dy\\
&=\frac{\pi^{\frac{d}{2}}|\Gamma(-s)|}{\Gamma\left(\frac{d+2s}{2}\right)}\Lambda_{d,\gamma,\al,\be,s}|x|^{-2s-\al-\be-\gamma}.
\end{align}
  Comparing the above equality with \eqref{eq:VeRd} we see that
  \[
  \mathcal{C}(d,s,2,\alpha,\beta) = \frac{\pi^{\frac{d}{2}}|\Gamma(-s)|}{\Gamma\left(\frac{d+2s}{2}\right)}\Lambda_{d,\gamma,\al,\be,s},
  \]
  and the result follows in the case $\al,\be,\al+\be\in(0,d-2s)$. For the general case we argue as follows.
  First, since $p=2$, in formula \eqref{C} we can replace $\left|1-r^{(d-\alpha-\beta-sp)/p}\right|^{p}$
    by $\left(1-r^{(d-\alpha-\beta-2s)/2}\right)^{2}$.
    After this change, we observe that both right hand sides of \eqref{C} and \eqref{Cds2} are holomorphic (as functions of two variables) with respect to $\alpha$ and $\beta$ in the domain $\{(\alpha,\beta)\in\mathbb{C}^2:\Ree(\alpha),\Ree(\beta),\Ree(\alpha+\beta)\in(-2s,d)\}$. Indeed, the integral in \eqref{C} is absolutely convergent, as we have $|r^{\al}+r^{\be}|\leq r^{\Ree(\al)}+r^{\Ree(\be)}$ and $|1-r^{z}|\sim |1-r^{\Ree z}|$, when $r\rightarrow 0^+$ or $r\rightarrow 1^{-}$. Holomorphicity then follows from Cauchy, Morera theorems and Osgood's lemma. Hence, since we have already shown that \eqref{Cds2} is valid for $\alpha,\beta,\alpha+\beta\in(0,d-2s)$, by the identity theorem for analytic functions it must be satisfied for all $\alpha,\beta,\alpha+\beta\in(-2s,d)$; see  Figure~\ref{fig:domainhol}.

  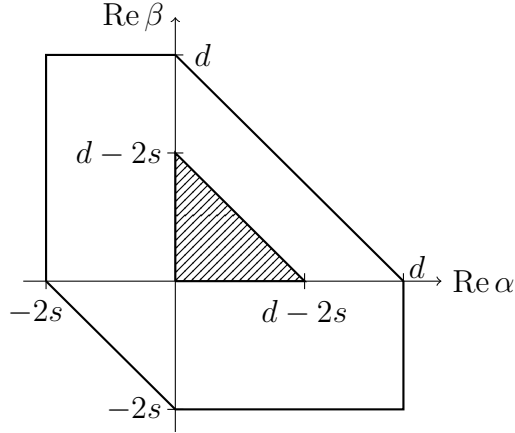
\begin{figure}
    \begin{center}
      \begin{tikzpicture}
        \draw[thin,->] (-2,0) -- (3.5,0) node[black,anchor=west] {$\Ree\alpha$};
        \draw[thin,->] (0,-2) -- (0,3.5) node[black,anchor=east] {$\Ree\beta$};
        \draw (-1.7 cm,3pt) -- (-1.7 cm,-3pt) node[anchor=north] {$-2s\ \ $};
        \draw (1.7 cm,3pt) -- (1.7 cm,-3pt) node[anchor=north] {${\textstyle{d-2s}}$};
        \draw (3.0 cm,3pt) -- (3.0 cm,-3pt) node[anchor=south west] {$\!d$};
        \draw (-3pt,3.0 cm) -- (3pt, 3.0cm) node[anchor=west] {$d$};
        \draw (-3pt,1.7cm) -- (3pt,1.7cm) node[anchor=east] {${{d-2s}}\ $};
        \draw (-3pt,-1.7 cm) -- (3pt, -1.7cm) node[anchor=east] {$-2s\ $};
        \draw[thick] (-1.7,0) -- (-1.7,3) -- (0,3) -- (3,0) -- (3, -1.7) -- (0,-1.7) -- (-1.7,0);
        \draw[thick,fill,pattern=north east lines] (0,0) -- (1.7,0) -- (0, 1.7) -- (0,0);        
  \end{tikzpicture}
  \end{center}
    \caption{
      The larger hexagonal region is
      the domain of holomorphicity, $\{(\alpha,\beta)\in\mathbb{C}^2:\Ree(\alpha),\Ree(\beta),\Ree(\alpha+\beta)\in(-2s,d)\}$, while the small triangle inside is the domain where the equality
    \eqref{Cds2}  is initially proved. See proof of Proposition~\ref{prop:C}.
    } \label{fig:domainhol}
\end{figure}

  Finally, we need to remove the assumption that $2s<d$. We may assume that $d=1$ as otherwise this assumption is always satisfied. We proceed similarly as before using the holomorphicity argument of
  both right hand sides of  \eqref{C} and \eqref{Cds2} as a~function of $s$, where $0<\Ree s<1$.
  For that we modify the right hand side of  \eqref{C} as above, and also the right hand side of \eqref{Cds2} by replacing $|\Gamma(-s)|$ by $-\Gamma(-s)$.
\end{proof}

\subsection{The case of the half-space}

In this subsection we denote
\[
\gamma=\frac{1+\alpha+\beta-sp}{p},\,w(x)=x_{d}^{-\gamma},\,k(x,y)=\frac{1}{2}|x-y|^{-d-sp}\left(x_{d}^{\alpha}y_{d}^{\beta}+x_{d}^{\beta}y_{d}^{\alpha}\right).
\]
Furthermore, we put    
\begin{equation}\label{D}
\mathcal{D}=\mathcal{D}(d,s,p,\alpha,\beta)=\frac{\pi^{\frac{d-1}{2}}\Gamma\left(\frac{1+sp}{2}\right)}{\Gamma\left(\frac{d+sp}{2}\right)}\int_{0}^{1}\frac{\left(t^{\alpha}+t^{\beta}\right)\left|1-t^{-\frac{1+\alpha+\beta-sp}{p}}\right|^{p}}{(1-t)^{1+sp}}\,dt.
\end{equation}

\begin{lem}\label{l.halfspace}
   Let $0<s<1$, $p\geq 1$ and $\alpha,\beta,\alpha+\beta\in (-1,sp)$
    and $1+\alpha+\beta\neq sp$.
 It holds
 \begin{equation}\label{eq:VePi}
 2\lim_{\varepsilon\rightarrow 0}\int_{|x_{d}-y_{d}|>\varepsilon,\,y_{d}>0}\left(w(x)-w(y)\right)|w(x)-w(y)|^{p-2}k(x,y)\,dy=\frac{\mathcal{D}(d,s,p,\alpha,\beta)}{x_{d}^{sp-\alpha-\beta}}w(x)^{p-1}
 \end{equation}
 uniformly on compacts sets contained in $\R^{d}_{+}$. The constant $\mathcal{D}(d,s,p,\al,\be)$ is given by \eqref{D}.
\end{lem}
\begin{proof}
Let $d\geq 2$. We have
\begin{align*}
2\int_{|x_{d}-y_{d}|>\varepsilon,\,y_{d}>0}\left(w(x)-w(y)\right)|w(x)-w(y)|^{p-2}k(x,y)\,dy=x_{d}^{\alpha}I_{\varepsilon}(\beta)+x_{d}^\beta I_{\varepsilon}(\alpha),     
\end{align*}
where 
$$
I_{\varepsilon}(a)=\int_{|x_{d}-y_{d}|>\varepsilon,\,y_{d}>0}\frac{\left(x_{d}^{-\gamma}-y_{d}^{-\gamma}\right)\left|x_{d}^{-\gamma}-y_{d}^{-\gamma}\right|^{p-2}y_{d}^{a}}{|x-y|^{d+sp}}\,dy.
$$
Let $x',y'\in\R^{d-1}$. By absolute convergence and the formula \cite[(6.2.1)]{MR1225604}, using the substitution $y'\mapsto x'+y'|x_d-y_d|,$ we may rewrite the integral $I_{\varepsilon}(a)$ for $a\in\{\alpha,\beta\}$ as 
\begin{align}
I_{\varepsilon}(a)&=\int_{|x_{d}-y_{d}|>\varepsilon,\,y_{d}>0}\left(x_{d}^{-\gamma}-y_{d}^{-\gamma}\right)\left|x_{d}^{-\gamma}-y_{d}^{-\gamma}\right|^{p-2}y_{d}^{a}\,dy_{d}\int_{\R^{d-1}}\frac{dy'}{\left(|x'-y'|^2+|x_{d}-y_{d}|^2\right)^{\frac{d+sp}{2}}}\nonumber\\
&=\frac{\pi^{\frac{d-1}{2}}\Gamma\left(\frac{1+sp}{2}\right)}{\Gamma\left(\frac{d+sp}{2}\right)}\int_{|x_{d}-y_{d}|>\varepsilon,\,y_{d}>0}\frac{\left(x_{d}^{-\gamma}-y_{d}^{-\gamma}\right)\left|x_{d}^{-\gamma}-y_{d}^{-\gamma}\right|^{p-2}y_{d}^{a}}{|x_{d}-y_{d}|^{1+sp}}\,dy_{d}\nonumber\\
&=\frac{\pi^{\frac{d-1}{2}}\Gamma\left(\frac{1+sp}{2}\right)}{\Gamma\left(\frac{d+sp}{2}\right)}x_{d}^{-\gamma(p-1)+a-sp}\int_{|1-t|>\varepsilon/x_{d},\,t>0}\frac{(1-t^{-\gamma})|1-t^{-\gamma}|^{p-2}t^{a}}{|1-t|^{1+sp}}\,dt.\label{eq.Iepsa}
\end{align}
To end the proof, it suffices to notice that
\begin{align*}
&\lim_{\varepsilon\rightarrow 0}\int_{|1-t|>\varepsilon/x_{d},\,t>0}\frac{(1-t^{-\gamma})|1-t^{-\gamma}|^{p-2}t^{a}}{|1-t|^{1+sp}}\,dt\\
  &=
 \, \lim_{\varepsilon\to 0} \left( \int_0^{1-\varepsilon} + \int_{1+\varepsilon}^\infty \right) 
\frac{(1-t^{-\gamma})|1-t^{-\gamma}|^{p-2}t^{a}}{|1-t|^{1+sp}}\,dt\\
  & =
     \lim_{\varepsilon\to 0} \left(
     \int_0^{1-\varepsilon} \frac{(1-t^{-\gamma})|1-t^{-\gamma}|^{p-2}t^{a}}{|1-t|^{1+sp}}\,dt 
     - \int_0^{\frac{1}{1+\varepsilon}} 
     \frac{(1-t^{-\gamma})|1-t^{-\gamma}|^{p-2}t^{\gamma(p-1)+sp-1-a}}{|1-t|^{1+sp}}\,dt \right)\\
&=\int_{0}^{1}\frac{\left(1-t^{-\gamma}\right)\left|1-t^{-\gamma}\right|^{p-2}\left(t^{a}-t^{\gamma(p-1)+sp-1-a}\right)}{(1-t)^{1+sp}}\,dt.
\end{align*}
Elementary calculations show that
\begin{align*}
&\left(1-t^{-\gamma}\right)\left|1-t^{-\gamma}\right|^{p-2}\left(t^{\alpha}-t^{\gamma(p-1)+sp-1-\alpha}+t^{\beta}-t^{\gamma(p-1)+sp-1-\beta}\right)\\
&=\left(t^{\alpha}+t^{\beta}\right)\left|1-t^{-\gamma}\right|^{p}.
\end{align*}
In the remaining case $d=1$
 the formula \eqref{eq.Iepsa} still holds, but its proof is much simpler. The rest of the proof
  remains unchanged and hence the result follows.
\end{proof}

\begin{prop}\label{prop:D}
For $p=2$ and $s\neq\frac{1}{2}$ the constant $\mathcal{D}$ reduces to 
\begin{align}
\mathcal{D}(d,s,2,\alpha,\beta)&=\frac{\pi^{\frac{d-1}{2}}\Gamma\left(\frac{1+2s}{2}\right)}{\Gamma\left(\frac{d+2s}{2}\right)}\Bigg[\BB(\alpha+1,-2s)+\BB(\beta+1,-2s)+\BB(2s-\alpha,-2s) \nonumber\\
&+\BB(2s-\beta,-2s)-2\BB\left(s+\frac{\alpha-\beta+1}{2},-2s\right)-2\BB\left(s+\frac{\beta-\alpha+1}{2},-2s\right)\Bigg], \label{eq:Dp2}
\end{align}
with $\BB$ being the Beta function.
In the case $p=2$ and $s=\frac{1}{2}$ we have
\[
\mathcal{D}(d,\tfrac{1}{2},2,\alpha,\beta)=\frac{\pi^{\frac{d+1}{2}}}{\Gamma\left(\frac{d+1}{2}\right)}\left((\alpha-\beta)\cot\frac{\pi(\alpha-\beta)}{2}-\alpha\cot\pi\alpha-\beta\cot\pi\beta\right),
\]
with the convention $0\cot0=1$.  For example, $\mathcal{D}(d,\tfrac{1}{2},2,0,0)=0$.
\end{prop}

In particular, when $\al=\be=0$, we obtain the value
\[
\mathcal{D}(d,s,2,0,0)=\frac{\pi^{\frac{d}{2}}\Gamma\left(\frac{1+2s}{2}\right)}{\Gamma\left(\frac{d+2s}{2}\right)}\frac{\BB\left(\frac{1+2s}{2},1-s\right)-4^{s}}{s4^{s}},
\]
established by Bogdan and Dyda in \cite{MR2663757}.

\begin{proof}[Proof of Proposition~\ref{prop:D}]
  We are going to prove the result by calculating the left hand side of \eqref{eq:VePi}
  in another way.
  First, we recall from
   \cite[(5.4) and (5.5)]{BBC} that
   \[
   Lx_d^{c}=\fgamm(2s,c)\frac{\pi^{\frac{d-1}{2}}\Gamma\left(\frac{1+2s}{2}\right)}{\Gamma\left(\frac{d+2s}{2}\right)}x_d^{c-2s}
   \]
   for $c\in(-1,2s)$, where $L=\mathcal{A}_{d,-2s}^{-1}\Delta_{\R^d_+}^{s}$ and
\begin{equation}\label{gamma}
\fgamm(a,b)=\int_{0}^{1}\frac{(t^{b}-1)(1-t^{a-b-1})}{(1-t)^{1+a}}\,dt,\quad a\in(0,2),\,b\in(-1,a).
\end{equation}
Morever, by \cite[(2.2)]{MR2663757} we have
\begin{align}\label{eq:fgamm}
  \fgamm(a,b)=\BB(b+1,-a)+\BB(a-b,-a)+\frac{1}{a}. 
\end{align}
Using an elementary equality like \eqref{eq:xaxb}, but with $x_d$ and $y_d$ in place of $|x|$ and $|y|$, respectively, we can proceed like in \eqref{JJ} to obtain that
\begin{align}\label{KK}
\nonumber
&\text{P.V.}\int_{\R^d_+}\left(x_d^{b}-y_d^{b}\right)\left(x_d^{\al}y_d^{\be}+x_d^{\be}y_d^{\al}\right)|x-y|^{-d-2s}\,dy\\
&=\frac{\pi^{\frac{d-1}{2}}\Gamma\left(\frac{1+2s}{2}\right)}{\Gamma\left(\frac{d+2s}{2}\right)}\left(\fgamm(2s,\al)+\fgamm(2s,\be)-\fgamm(2s,\al+b)-\fgamm(2s,\be+b)\right)x_d^{-2s+b+\al+\be},
\end{align}
 where $b= - \gamma$ and $\al,\be,\al+\be\in(-1,2s)$.
 Comparing \eqref{KK} with \eqref{eq:VePi} we conclude that
\begin{align*}
\mathcal{D}(d,s,2,\al,\be)&=\frac{\pi^{\frac{d-1}{2}}\Gamma\left(\frac{1+2s}{2}\right)}{\Gamma\left(\frac{d+2s}{2}\right)}\Bigg[\fgamm(2s,\al)+\fgamm(2s,\be)-\fgamm\left(2s,s+\frac{\al-\be+1}{2}\right)\\
&-\fgamm\left(2s,s+\frac{\be-\al+1}{2}\right)\Bigg],
\end{align*}
therefore \eqref{eq:Dp2} follows by an application of \eqref{eq:fgamm}.

The limit case $\lim_{s\rightarrow\frac{1}{2}}\mathcal{D}(d,s,2,\alpha,\beta)$ can be obtained as follows. We have 
\begin{align*}
\mathcal{D}(d,s,2,\alpha,\beta)&=\frac{\pi^{\frac{d-1}{2}}\Gamma\left(\frac{1+2s}{2}\right)\Gamma(-2s)}{\Gamma\left(\frac{d+2s}{2}\right)}\Bigg[\frac{\Gamma(\alpha+1)}{\Gamma(\alpha+1-2s)}+\frac{\Gamma(\beta+1)}{\Gamma(\beta+1-2s)}+\frac{\Gamma(2s-\alpha)}{\Gamma(-\alpha)}\\
&+\frac{\Gamma(2s-\beta)}{\Gamma(-\beta)}-2\frac{\Gamma\left(s+\frac{\alpha-\beta+1}{2}\right)}{\Gamma\left(-s+\frac{\alpha-\beta+1}{2}\right)}-2\frac{\Gamma\left(s+\frac{\beta-\alpha+1}{2}\right)}{\Gamma\left(-s+\frac{\beta-\alpha+1}{2}\right)}\Bigg].
\end{align*} 
The derivative of the expression in the square brackets above with respect to $s$ at $s=\frac{1}{2}$ is given by
\begin{align*}
&2\Bigg[\alpha\psi(\alpha)-\alpha\psi(1-\alpha)+\beta\psi(\beta)-\beta\psi(1-\beta)\\
&-\frac{\alpha-\beta}{2}\left(\psi\left(1+\frac{\alpha-\beta}{2}\right)+\psi\left(\frac{\alpha-\beta}{2}\right)\right)\\
&-\frac{\beta-\alpha}{2}\left(\psi\left(1+\frac{\beta-\alpha}{2}\right)+\psi\left(\frac{\beta-\alpha}{2}\right)\right)\Bigg],
\end{align*}
where $\psi=\Gamma'/\Gamma$ is the digamma function. Since $\Gamma(-2s)=\frac{\Gamma(2-2s)}{-2s(1-2s)}$, the result follows from L'Hôpital's rule, the functional equation $\psi(1+z)=\psi(z)+1/z$ and the reflection formula $\psi(1-z)-\psi(z)=\pi\cot \pi z$.
\end{proof}

\subsection{Limiting cases when $s\to 0^+$}
\begin{prop}
It holds
\begin{equation}\label{C0}
\mathcal{C}(d,0,p,\al,\be)=
\frac{2\pi^{\frac{d}{2}}}{\Gamma\left(\frac{d}{2}\right)}\int_{0}^{1}\frac{\left(r^{\al-1}+r^{\be-1}\right)\left(1-r^{\frac{d-\al-\be}{p}}\right)^p}{1-r^2}\,dr.
\end{equation}

When $p=2$, one has
\begin{align*}
\mathcal{C}(d,0,2,\al,\be)&=\frac{\pi^{\frac{d}{2}}}{\Gamma\left(\frac{d}{2}\right)}\Bigg[2\psi\left(\frac{d-\al+\be}{4}\right)+2\psi\left(\frac{d+\al-\be}{4}\right)\\
&-\psi\left(\frac{\al}{2}\right)-\psi\left(\frac{\be}{2}\right)-\psi\left(\frac{d-\al}{2}\right)-\psi\left(\frac{d-\be}{2}\right)\Bigg]
\end{align*}
and
\begin{align*}
\mathcal{D}(d,0,2,\al,\be)&=\frac{\pi^{\frac{d}{2}}}{\Gamma\left(\frac{d}{2}\right)}\Bigg[2\psi\left(\frac{\al-\be+1}{2}\right)+2\psi\left(\frac{\be-\al+1}{2}\right)\\
&-\psi(\al+1)-\psi(-\al)-\psi(\be+1)-\psi(-\be)\Bigg],
\end{align*}
with $\psi$ being the digamma function.
\end{prop}

\begin{proof}
  The formula \eqref{C0} follows from \eqref{C}, \eqref{Phi} and \eqref{PhiFF},
  by noting that for $d\geq 2$ it holds
  $\Phi_{d,0,p}(r) = \left|\mathbb{S}^{d-1}\right|  \FF\left(\frac{d}{2},1;\frac{d}{2};r^2\right) = \left|\mathbb{S}^{d-1}\right| (1-r^2)^{-1}$.
  
  The proofs of the remaining equalities follow directly from previous results, form of the constants $\mathcal{C}(d,s,p,\al,\be)$, $\mathcal{D}(d,s,p,\al,\be)$ and basic calculations and will be omitted.
\end{proof}

\section{Proof of the weighted fractional Hardy inequality on $\R^d$}

\begin{proof}[Proof of Theorem \ref{tw1}]
We substitute $u(x)=v(x)w(x)$ and use the inequality \eqref{generalhardy} and Lemma \ref{lem.w1} to conclude that for $u\in C_c(\R^d)$ we have the inequality 
\begin{align*}
\int_{\R^d}\int_{\R^d}\frac{|u(x)-u(y)|^{p}}{|x-y|^{d+sp}}|x|^{-\al}|y|^{-\be}\,dy\,dx&\geq\mathcal{C}\int_{\R^d}\frac{|u(x)|^{p}}{|x|^{sp+\al+\be}}\,dx,\\
\end{align*}
where $\mathcal{C}$ is given by \eqref{C}. Hence, \eqref{hardy1} holds. It suffices to show that the constant $\mathcal{C}$ is optimal. This fact is possible to obtain by slightly modifying the proof of the sharpness from \cite{MR2469027} and also \cite[Theorem 2.9]{MR3626031}, when $\al=\be$, and we omit the details.
\end{proof}

\begin{proof}[Proof of Theorem \ref{tw2}]
This is a straightforward consequence of Lemma \ref{lem.w1} and \eqref{generalhardyremainder}.
\end{proof}

\section{Proof of the weighted fractional Hardy inequality on $\R^{d}_{+}$}\label{Section.halfspace}

\begin{proof}[Proof of Theorem~\ref{tw3}]
Let $u\in C_{c}(\R^{d}_{+})$. We substitute $u=wv$, where $w(x)=x_{d}^{-\gamma}$. By Lemma \ref{l.halfspace} and \cite[Proposition 2.2]{MR2469027} we have
\begin{align*}
\int_{\R^{d}_{+}}\int_{\R^{d}_{+}}\frac{|u(x)-u(y)|^{p}}{|x-y|^{d+sp}}\,x_{d}^{\alpha}\,y_{d}^{\beta}\,dy\,dx&\geq\mathcal{D}(d,s,p,\alpha,\beta)\int_{\R^{d}_{+}}\frac{|u(x)|^{p}}{x_{d}^{sp-\alpha-\beta}}\,dx
\end{align*}
where $\mathcal{D}(d,s,p,\alpha,\beta)$ is given by (\ref{D}). Hence, it suffices to show that the constant $\mathcal{D}(d,s,p,\alpha,\beta)$ is optimal. To do this we first assume that $d=1$. Then the optimality is easy to achieve by a slight modification of the proof for the whole space $\R$, with the same approximating functions; we refer again to \cite{MR2469027} and \cite{MR3626031}. For greater dimensions $d\geq 2$ we allege the argument of Frank and Seiringer from \cite{MR2723817}. Namely, for $x'\in\R^{d-1}$ we define the functions $u_n(x)=\chi_n(x')\varphi(x_d)$, where
$$
\chi_n(x')=\begin{cases}
 1 ~~~~~~~~~~~~~~~~~\text{if\,} |x'|\leq n,\\
 n+1-|x'|~~~\text{\,if\,}n<|x'|<n+1,\\
 0~~~~~~~~~~~~~~~~~\text{if\,}|x'|\geq n+1.
\end{cases}
$$
We have 
\begin{align*}
\lim_{n\rightarrow\infty}\frac{\displaystyle\int_{\R^d_+}\int_{\R^d_+}\frac{|u_n(x)-u_n(y)|^{p}}{|x-y|^{d+sp}}\,x_d^{\al}\,y_d^{\be}\,dy\,dx}{\displaystyle\int_{\R^d_+}\frac{|u_n(x)|^p}{x_d^{sp-\al-\be}}\,dx}=A\frac{\displaystyle\int_{0}^{\infty}\int_{0}^{\infty}\frac{|\varphi(x_d)-\varphi(y_d)|^p}{|x_d-y_d|^{1+sp}}\,x_d^{\al}\,y_d^{\be}\,dx_d\,dy_d}{\displaystyle\int_{0}^{\infty}\frac{|\varphi(x_d)|^p}{x_d^{sp-\al-\be}}\,dx_d},    
\end{align*}
where $A=\frac{\mathcal{D}(d,s,p,\al,\be)}{\mathcal{D}(1,s,p,\al,\be)}$. Hence, sharpness of the constant $\mathcal{D}(d,s,p,\al,\be)$ for $d\geq 2$ follows from the sharpness of the constant $\mathcal{D}(1,s,p,\al,\be)$ for $d=1$.
\end{proof}
\begin{proof}[Proof of Theorem \ref{tw4}]
Follows directly from \eqref{generalhardyremainder} and Lemma \ref{l.halfspace}. 
\end{proof}
\begin{proof}[Proof of Theorem \ref{tw5}]
According to Theorem \ref{tw4}, it suffices to show that the remainder term dominates with the constant the integral $\left(\int_{\R^d_+}|u(x)|^{q}x_{d}^{\frac{q}{p}(\al+\be)}\,dx\right)^{\frac{p}{q}}$. To do this, we modify the proof of the unweighted Hardy--Sobolev--Maz'ya inequality from \cite[Section 2]{MR2910984}. 

We have 
\[
x_d^{-\frac{1-\al+\be-sp}{2}}y_d^{-\frac{1+\al-\be-sp}{2}}=\left(x_d y_d\right)^{\frac{sp-1}{2}}\left(\frac{x_d}{y_d}\right)^{\frac{\al-\be}{2}}.
\]
Without loss of generality, we may and do assume that $\al\geq\be$. Hence, recalling that $v(x)=x_d^{\frac{1+\al+\be-sp}{p}}u(x)$, we make the estimation
\begin{align*}
&\int_{\R^{d}_{+}}\int_{\R^{d}_{+}}\frac{|v(x)-v(y)|^{p}}{|x-y|^{d+sp}}x_{d}^{-\frac{1-\alpha+\beta-sp}{2}}y_{d}^{-\frac{1+\alpha-\beta-sp}{2}}\,dy\,dx\\
&\geq \iint_{\{x_d> y_d\}}\frac{|v(x)-v(y)|^{p}}{|x-y|^{d+sp}}\left(x_d y_d\right)^{\frac{sp-1}{2}}\left(\frac{x_d}{y_d}\right)^{\frac{\al-\be}{2}}\,dy\,dx\\
&\geq \iint_{\{x_d> y_d\}}\frac{|v(x)-v(y)|^{p}}{|x-y|^{d+sp}}\left(x_d y_d\right)^{\frac{sp-1}{2}}\,dy\,dx\\
&=\frac{1}{2}\int_{\R^{d}_{+}}\int_{\R^{d}_{+}}\frac{|v(x)-v(y)|^{p}}{|x-y|^{d+sp}}\left(x_d y_d\right)^{\frac{sp-1}{2}}\,dy\,dx,
\end{align*}
where in the last passage we used the symmetry of the integrand. Now, we repeat the computations from \cite{MR2910984} to conclude that the last term is an upper bound for $\left(\int_{\R^d_+}|u(x)|^{q}x_{d}^{\frac{q}{p}(\al+\be)}\,dx\right)^{\frac{p}{q}}$. The proof is complete.
\end{proof}

The next result is the first step in proving the weighted fractional Hardy inequality for general domains, see an unweighted analogue \cite[Theorem 2.5]{LossSloane}.
\begin{lem}\label{l.open}
Let $J\subset\R$ be a nonempty open set, $\al,\be,\al+\be\in(-1,sp)$ and $sp>1+\al+\be$. Then for all $u\in C_c(J)$,
\begin{equation}\label{J}
\int_{J}\int_{J}\frac{|u(x)-u(y)|^{p}}{|x-y|^{1+sp}}d_J(x)^{\al} d_J(y)^{\be}\,dy\,dx\geq\mathcal{D}(1,s,p,\al,\be)\int_{J}\frac{|u(x)|^{p}}{d_J(x)^{sp-\al-\be}}\,dx,   
\end{equation}
where $d_J(x)=\dist(x,\partial J)$.
\end{lem}
\begin{proof}
We first prove the inequality \eqref{J} for $J=(0,1)$. The following proof is a modification of \cite[Proof of Theorem 2.5]{LossSloane}. For $x\in(0,1)$ and $\varepsilon>0$ we define 
$$
V_{\varepsilon}(x)=2w(x)^{-p+1}\left(\int_{0}^{x-\varepsilon}+\int_{x+\varepsilon}^{1}\right)\left(w(x)-w(y)\right)\left|w(x)-w(y)\right|^{p-2}k(x,y)\,dy,
$$
recalling that $k(x,y)=\frac{1}{2}|x-y|^{-1-\al}\left(x^{\al}y^{\be}+x^{\be}y^{\al}\right)$ and $w(x)=x^{-\frac{1+\al+\be-sp}{p}}$. We observe that 
\begin{align*}
 V(x)&:=\lim_{\varepsilon\rightarrow 0^+}V_{\varepsilon}(x)\\
&=2w(x)^{-p+1}\text{P.V.}\left(\int_{0}^{\infty}-\int_{1}^{\infty}\right)\left(w(x)-w(y)\right)\left|w(x)-w(y)\right|^{p-2}k(x,y)\,dy\\
    &\geq2w(x)^{-p+1}\text{P.V.}\int_{0}^{\infty}\left(w(x)-w(y)\right)\left|w(x)-w(y)\right|^{p-2}k(x,y)\,dy\\
    &=\mathcal{D}(1,s,p,\al,\be)x^{-sp+\al+\be},
\end{align*}
by Lemma \ref{l.halfspace}, since $w(x)\leq w(y)$ for $y\in[1,\infty)$. In addition, by the uniform convergence, for any compact $K\subset (0,1]$ and small $\varepsilon$, $V_{\varepsilon}(x)>0$ uniformly for all $x\in K$. Hence, by \cite[(2.17)]{MR2469027} with $(0,1)$ instead of $\R^d$ and Fatou's lemma,
\begin{equation}\label{01}
  \int_{0}^{1}\int_{0}^{1}\frac{|v(x)-v(y)|^{p}}{|x-y|^{1+sp}}x^{\al}y^{\be}\,dy\,dx\geq\mathcal{D}(1,s,p,\al,\be)\int_{0}^{1}\frac{|v(x)|^{p}}{x^{sp-\al-\be}}\,dx,
\end{equation}
for any bounded function $v$ for with $\supp v\subset (0,1]$. 
By scaling, an analogous inequality is satisfied for any interval $(a,b)$. Moreover, using \eqref{01}, we have for $u\in C_c(J)$,
\begin{align*}
&\int_{0}^{1}\int_{0}^{1}\frac{|u(x)-u(y)|^{p}}{|x-y|^{1+sp}}\min\{x,1-x\}^{\al}\min\{y,1-y\}^{\be},dy\,dx\\
&\geq\int_{0}^{\frac{1}{2}}\int_{0}^{\frac{1}{2}}\frac{|u(x)-u(y)|^{p}}{|x-y|^{1+sp}}x^{\al}y^{\be}\,dy\,dx\\
&+\int_{\frac{1}{2}}^{1}\int_{\frac{1}{2}}^{1}\frac{|u(x)-u(y)|^{p}}{|x-y|^{1+sp}}(1-x)^{\al}(1-y)^{\be}\,dy\,dx\\
&=\int_{0}^{\frac{1}{2}}\int_{0}^{\frac{1}{2}}\frac{|u(x)-u(y)|^{p}}{|x-y|^{1+sp}}x^{\al}y^{\be}\,dy\,dx\\
&+\int_{0}^{\frac{1}{2}}\int_{0}^{\frac{1}{2}}\frac{|u(1-x)-u(1-y)|^{p}}{|x-y|^{1+sp}}x^{\al}y^{\be}\,dy\,dx\\
&\geq\mathcal{D}(1,s,p,\al,\be)\int_{0}^{\frac{1}{2}}\frac{|u(x)|^{p}+|u(1-x)|^{p}}{x^{sp-\al-\be}}\,dx\\
&=\mathcal{D}(1,s,p,\al,\be)\int_{0}^{1}\frac{|u(x)|^{p}}{\min\{x,1-x\}^{sp-\al-\be}}\,dx.
\end{align*}
That proves \eqref{J} for $J=(0,1)$. A general result follows again from scaling and translating and the fact that every nonempty, open subset of $\R$ is a countable union of disjoint intervals.
\end{proof}
\begin{proof}[Proof of Theorem \ref{tw7}]
  We repeat the proof of Loss and Sloane and reduce the problem to one dimension. By $\mathcal{L}_{\omega}$ we denote the $(d-1)$-Lebesgue measure on the plane $\{x\cdot\omega=0\}$. By Loss--Sloane formula \cite[Lemma~2.4]{LossSloane}\footnote{ The formula \cite[Lemma~2.4]{LossSloane} is stated for an integral $\int_\Omega\int_\Omega \frac{|f(x)-f(y)|^p}{|x-y|^{d+sp}}\,dy\,dx$ with $f\in C_c^\infty(\Omega)$
    and therefore cannot be directly applied. However, it is easy to see that the version we need also holds by the same argument.}  and \eqref{J}, we have
\begin{align*}
&\int_{\Omega}\int_{\Omega}\frac{|u(x)-u(y)|^{p}}{|x-y|^{d+sp}}d_{\Omega}(x)^{\al}d_{\Omega}(y)^{\be}\,dy\,dx&\\
&=\frac{1}{2}\int_{\mathbb{S}^{d-1}}\,d\omega\int_{\{x:x\cdot\omega=0\}}\,d\mathcal{L}_{\omega}(x)\int_{\{x+s\omega\in \Omega\}}\,ds\int_{\{x+t\omega\in \Omega\}}\,dt\\
&\times\frac{|u(x+s\omega)-u(x+t\omega)|^{p}}{|s-t|^{1+sp}}d_{\Omega}(x+s\omega)^{\al}d_\Omega(x+t\omega)^{\be}\\
&\geq \frac{\mathcal{D}(1,s,p,\al,\be)}{2}\int_{\mathbb{S}^{d-1}}\,d\omega\int_{\{x:x\cdot\omega=0\}}\int_{\{x+t\omega\in \Omega\}}\frac{|u(x+t\omega)|^{p}}{d_{\omega}(x+t\omega)^{sp-\al-\be}}\,dt\,d\mathcal{L}_{\omega}(x)\,d\omega\\
&=\mathcal{D}(d,s,p,\al,\be)\int_{\Omega}\frac{|u(x)|^{p}}{m_{sp-\al-\be}(x)^{sp-\al-\be}}\,dx,
\end{align*}
where the last inequality follows from the calculation
$$
\int_{\mathbb{S}^{d-1}}|\omega_{d}|^{sp}\,d\omega =\frac{2 \pi^{\frac{d-1}{2}}\Gamma\left(\frac{1+sp}{2}\right)}{\Gamma\left(\frac{d+sp}{2}\right)}
$$ and the fact that $d_{\Omega}(x+s\omega)^{\al}d_\Omega(x+t\omega)^{\be} \geq d_{\omega}(x+s\omega)^{\al}d_\omega(x+t\omega)^{\be}$. That proves \eqref{hardyconvex1}. The inequality \eqref{hardyconvex2} follows from the fact that $m_a(x)\leq d_\Omega(x)$, if $\Omega$ is convex (we refer again to \cite{LossSloane}). Optimality of the constant $\mathcal{D}$ in \eqref{hardyconvex2} can be obtained similarly as in the case of the half-space, by picking a hyperplane tangent to $\Omega$ at some point $d_0\in\partial \Omega$ and transplainting the approximating functions near $d_0$. We left the details to the Reader.
\qedhere
\end{proof}


\def\cprime{$'$} \def\cprime{$'$}
  \def\polhk#1{\setbox0=\hbox{#1}{\ooalign{\hidewidth
  \lower1.5ex\hbox{`}\hidewidth\crcr\unhbox0}}}
  \def\polhk#1{\setbox0=\hbox{#1}{\ooalign{\hidewidth
  \lower1.5ex\hbox{`}\hidewidth\crcr\unhbox0}}}
  \def\polhk#1{\setbox0=\hbox{#1}{\ooalign{\hidewidth
  \lower1.5ex\hbox{`}\hidewidth\crcr\unhbox0}}} \def\cprime{$'$}
  \def\cprime{$'$} \def\cprime{$'$} \def\cprime{$'$} \def\cprime{$'$}
  \def\cprime{$'$}

\end{document}